\documentclass{amsart}

\usepackage[dvipsnames,svgnames,x11names,hyperref]{xcolor}
\usepackage{bbm,url,graphicx,verbatim,amssymb,enumerate,booktabs,lmodern,mathtools,microtype,mathabx}
\usepackage[pagebackref,colorlinks,citecolor=Mahogany,linkcolor=Mahogany,urlcolor=Mahogany,filecolor=Mahogany]{hyperref}
\usepackage[capitalize]{cleveref}
\usepackage{amsmath,amscd}
\usepackage[mathscr]{euscript}
\usepackage[all]{xy}
\usepackage[margin=1.5in]{geometry}
\usepackage{tikz,tikz-cd}
\usetikzlibrary{matrix,calc,positioning,arrows,decorations.pathreplacing,patterns,arrows}

\newtheorem{theorem}{Theorem}[section]
\newtheorem{lemma}[theorem]{Lemma}
\newtheorem{proposition}[theorem]{Proposition}
\newtheorem{corollary}[theorem]{Corollary}
\newtheorem*{corollary*}{Corollary}

\theoremstyle{definition}
\newtheorem{definition}[theorem]{Definition}

\theoremstyle{remark}
\newtheorem{example}[theorem]{Example}

\newtheorem{remark}[theorem]{Remark}

\newcommand{\cat}[1]{{\mathbf{#1}}}

\renewcommand{\rm}[1]{{\mathrm{#1}}}

\newcommand{\op}{\rm{op}}

\newcommand{\Hom}{\rm{Hom}}
\newcommand{\Alg}{\cat{Alg}}
\newcommand{\id}{\rm{id}}

\newcommand{\ZZ}{\mathbb{Z}}
\newcommand{\QQ}{\mathbb{Q}}
\newcommand{\RR}{\mathbb{R}}

\newcommand{\FF}{\mathbb{F}}

\newcommand{\AQ}{\rm{AQ}}
\newcommand{\rlie}{\rm{rLie}}
\newcommand{\alg}{\cat{Alg}}
\renewcommand{\L}{\rm{L}}
\newcommand{\sym}{\rm{Sym}}
\newcommand{\W}{\rm{W}}

\title{$E_2$-formality via obstruction theory}

\author{Geoffroy Horel}
\address{Université Sorbonne Paris Nord, Laboratoire Analyse, Géométrie et Applications, CNRS (UMR 7539), 93430, Villetaneuse, France.}
\email{horel@math.univ-paris13.fr}

\date{\today}

\begin{document}

\maketitle
	
\begin{abstract}
We attack the question of $E_2$-formality of differential graded algebras over $\FF_p$ via obstruction theory. We are able to prove that $E_2$-algebras whose cohomology ring is a polynomial algebra on even degree classes are formal. As a consequence we prove $E_2$-formality of the classifying space of some compact Lie groups or of Davis-Januszkiewicz spaces.
\end{abstract}

Formality of spaces is an old idea originating in the field of rational homotopy theory. In that context, a space is said to be formal if its rational cohomology is quasi-isomorphic to its cochains as a commutative or $E_\infty$-algebra. When this is the case, the whole rational homotopy type is controlled by a very manageable algebraic gadget.

With integral or torsion coefficients, the question of formality admits several versions. The most naive generalization (i.e. asking for cochains to be quasi-isomorphic to cohomology as $E_\infty$-algebras) does not have interesting examples. Indeed, it is an observation of Mandell that cochains on a space can never be quasi-isomorphic to a strictly commutative algebra unless the space is homotopically discrete (see \cite[Proposition 5.1]{flynnobstruction} for a formal proof of this observation). Therefore, one is led to study weaker forms of formality. There is some literature devoted to proving $E_1$-formality of certain spaces, i.e. proving that $C^*(X,k)$ is quasi-isomorphic to $H^*(X,k)$ as differential graded algebras (see for example \cite{elhaouari,berglundkoszul,salvatoreformality,drummondhorel,ciricietale}). When this is the case, some invariants of $X$ can be computed from the cohomology ring of $X$. For example, the bar construction spectral sequence
\[\rm{Tor}^{H^*(X;\FF_p)}(\FF_p,\FF_p)\Rightarrow H^*(\Omega X,\FF_p)\]
collapses at the $E_2$-page. However, this collapse result is additive and there are usually some multiplicative extension, that cannot be resolved by $E_1$-formality.

There is in fact a countable family of formality properties interpolating between $E_1$-formality and $E_\infty$-formality. Namely, one can study $E_n$-formality for any $1\leq n\leq \infty$. We say that a space $X$ is $E_n$-formal if its cohomology ring viewed as an $E_n$-algebra using the map of operads $E_n\to\mathcal{C}om$ is quasi-isomorphic to its singular cochains with its underlying $E_n$-algebra structure. This notion was introduced by Mandell in \cite{mandellformality}. He made several conjectures about it. In particular, he conjectured that $n$-fold suspensions are $E_n$-formal. This conjecture was proved very recently in \cite{heutsland}.

In the present paper, we study $E_2$-formality. Our approach is obstruction theoretic. Obstruction theory has been classically used to prove or study formality (see for example \cite{halperinstasheff,berglundkoszul,salehformality,emprinkaledin}). Typically, there is a sequence of obstruction classes in Hochschild or André-Quillen cohomology groups that have to vanish for the algebra to be formal. In general, computing the actual obstructions can be very difficult unless the group in which they live is zero (see \cite[Proposition 6.9]{salvatoreformality} for an example of a non-trivial obstruction class actually being computed).

In our case, we exploit the fact that the operations on the homology of an $E_2$-algebras are very explicit, thanks to the work of Cohen (see \cite{cohenladamay}), and quite manageable. The obstruction group is a Quillen cohomology group in the category of $\W_1$-algebras, where the monad $\W_1$ is the monad of homology operations on $E_2$-algebras. This makes the obstruction groups computable in certain easy situations. Our main result is Theorem \ref{theorem : main} which states that $E_2$-algebras whose cohomology ring is polynomial on even degree variables are $E_2$ formal.

We are in fact able to push this result a bit further to prove formality of certain diagrams of $E_2$-algebras. As a corollary, we prove $E_2$-formality of the classifying space of compact Lie groups at primes that do not divide the order of the Weyl group (Theorem \ref{theorem : formality compact Lie group}) generalizing a recent result of Benson and Greenlees proving $E_1$-formality of such spaces (see \cite{bensonformality}). We recover $E_2$-formality of Davis-Januszkiewicz spaces (originally proven by Franz in \cite{franzhomotopy}) and we prove a multiplicative collapse result for certain Eilenberg-Moore spectral sequences (Theorem \ref{theorem : formality and EMSS}). 

\subsection*{Acknowledgements}
 I am grateful to Jeffrey Carlson, Alexander Berglund, and Anibal Medina-Mardones for useful conversations about this paper. I also thank the anonymous referee for their careful reading of my manuscript and for making many useful suggestions that improved the clarity of this paper.

\subsection*{Conventions}

We denote by $\cat{grVect}_k$ the category of graded vector spaces over a field $k$. This category has a symmetric monoidal structure whose symmetry isomorphism involves the usual sign. Our graded vector spaces are cohomologically graded. We write $V\mapsto sV$ for the shift in this category given by $(sV)^i=V^{i-1}$.

Given a graded vector space $V$, we denote by $\sym(V)$ the symmetric algebra on $V$, i.e. the free graded commutative algebra on $V$. If $V$ is concentrated in odd degrees, we write $\Lambda(V)$ for the exterior algebra on $V$ which is, by definition, the free \emph{strictly} commutative algebra on $V$ (strictly commutative means that elements of odd degrees square to zero). Of course, by the sign rules, there is an isomorphism $\Lambda(V)\cong \sym(V)$ if the characteristic of the field is different from $2$.

\section{Quillen cohomology}

\subsection{Non-abelian derived functor}

We follow the treatment of \cite{franklandbehavior}. Let $\cat{C}$ be a complete and cocomplete category with a set of projective compact generators. Following Frankland, such a category shall be called quasi-algebraic. Thanks to a theorem of Quillen (see \cite[Theorem 2.4]{quillenhomotopical}), if $\cat{C}$ is a quasi-algebraic category, its category of simplicial objects, denoted $s\cat{C}$ has a model structure transferred along the right adjoint functor
\begin{align*}
s\cat{C}&\to s\cat{Set}^{\mathcal{G}}\\
X_\bullet&\mapsto \{\Hom_{\cat{C}}(G,X_\bullet)\}_{G\in\mathcal{G}}
\end{align*}
where $\mathcal{G}$ is a set of compact projective generators of $\cat{C}$.

Given a left adjoint functor between two quasi-algebraic categories, we define its left derived functor
\[\L F\colon s\cat{C}\to s\cat{D}\]
to be the left derived functor in the sense of model categories of the functor $F\colon s\cat{C}\to s\cat{D}$ given by degreewise application of the functor $F$.

\begin{remark}\label{remark : non-abelian derived}
There is a conceptural description of the $\infty$-category underlying the model category $s\cat{C}$ as the \emph{non-abelian derived category of} $\cat{C}$. Explicitly, this is completion of the category $\cat{C}^{cp}$ of compact projective objects of $\cat{C}$ under $\infty$-sifted colimits. Using this terminology, the functor $\L F$ is the unique extension of $F_{|\cat{C}^{cp}}$ into a functor that preserves sifted colimits. This is originally due to Lurie (see \cite[Section 5.5.8]{luriehigher} or  \cite[Section 5.1.1]{kestupurity} for a concise version of Lurie's work).
\end{remark}

\subsection{Quillen cohomology}

Let $\cat{C}$ be a quasi-algebraic category. For $c$ an object of $\cat{C}$, we may consider the category $\cat{Ab}(\cat{C}/c)$ of abelian group objects over $c$. By the adjoint functor theorem, there is an abelianization functor
\[\rm{Ab}\colon \cat{C}/c\to\cat{Ab}(\cat{C}/c)\]
which is left adjoint to the forgetful functor. We define the cotangent complex of $c$ to be the left derived functor of the abelianization functor. We denote this object by $L_c^{\cat{C}}$ or simply $L_c$ if there is no ambiguity. It follows from \cite[Propositions 3.33 and 3.34]{franklandbehavior} that the categories $\cat{C}/c$ and $\rm{Ab}(\cat{C}/c)$ are quasi-algebraic and hence that the relevant model structures exist.

Given an object $m\in\cat{Ab}(\cat{C}/c)$, we define the Quillen cohomology of $c$ with coefficients in $m$ as~:
\[\AQ^s_{\cat{C}}(c,m):=\pi^s\Hom_{\cat{Ab}(\cat{C}/c)}(L_c,m).\]
(Observe that $\Hom_{\cat{Ab}(\cat{C}/c)}(L_c,m)$ is a cosimplicial abelian group, and we denote by $\pi^s$ the $s$-th cohomology group of its associated cochain complex.)

\begin{remark}
In all the cases of interest to us, the category $\cat{C}$ will come with a forgetful functor to the category of graded vector spaces over $k$. Moreover, it will be the case that the shift functor on the category of graded vector space will pass to the category of abelian group objects in $\cat{C}/c$. In this case, the André-Quillen cohomology groups are bigraded
\[\AQ^{s,t}_{\cat{C}}(c,m)=\AQ^s_{\cat{C}}(c,s^tm).\]
\end{remark}

\begin{proposition}\label{prop : base change}
Let $F\colon\cat{C}\leftrightarrows \cat{D}\colon U$ be an adjunction between quasi-algebraic categories. Let $c$ be an object of $\cat{C}$. Assume that $\L F(c)\to F(c)$ is a weak equivalence. Let $m\in\cat{Ab}(\cat{D}/F(c))$. Then there is an isomorphism
\[\AQ^*_{\cat{D}}(Fc,m)\cong\AQ^*_{\cat{C}}(c,Um).\]
\end{proposition}

\begin{proof}
This is almost \cite[Proposition 4.10 (4)]{franklandbehavior} except that we do not ask that $F$ preserves all weak equivalences contrary to Frankland. However, it is straightforward to check that all that is needed in the proof is that $\L F(c)\to F(c)$ is a weak equivalence.
\end{proof}

\section{Cohomology of $E_2$-algebras}

In this section, we restrict to working over a prime field $\FF_p$ with $p$ a prime number. We recall the work of Cohen (this was originally published in \cite[Chapter III]{cohenladamay} but some signs were fixed in \cite[Section 16]{galatiuscellular}) describing the cohomology operations on $E_2$-algebras.

\begin{definition}
A shifted restricted Lie algebra is a graded vector space $V^*$ with a Lie bracket
\[[-,-]\colon V^i\otimes V^j\to V^{i+j-1}\]
and a restriction
\[\xi\colon V^i\to V^{pi-p+1}\]
satisfying the following relations
\begin{enumerate}
\item The Lie bracket is bilinear.
\item The Lie bracket is antisymmetric
\[[x, y] = -(-1)^{(|x|-1)(|y|-1)}[y, x].\]
\item The Lie bracket satisfies the graded Jacobi relation.
\[(-1)^{(|x|-1)(|z|-1)}[x, [y, z]]
+ (-1)^{(|x|-1)(|y|-1)}[y, [z, x]]
+ (-1)^{(|y|-1)(|z|-1)}[z, [x, y]]=0\]
\item The Lie bracket satisfies the relation $[x,[x,x]]=0$. (This relation follows from the Jacobi relation if $p\neq 3$.)
\item If $p$ is odd, the operation $\xi$ is zero on even degree elements. 
\item We have
\[\xi(x+y)=\xi(x)+\xi(y)+\sum_{i=1}^{p-1}d_2^i(x)(y)\]
where the operations $d_2^i$ are described in \cite[p.218]{cohenladamay}.
\item We have
\[\xi(\lambda x)=\lambda\xi(x)\]
for all $\lambda\in\FF_p$.
\item We have 
\[[x,\xi(y)]=\rm{ad}^p(y)(x)\]
\end{enumerate}
\end{definition}

\begin{remark}
Over $\FF_2$, relation (6) becomes
\[\xi(x+y)=\xi(x)+\xi(y)+[x,y].\]
One consequence of this is that $[x,x]=0$ whatever the degree of $x$ is. In other characteristics this only holds for elements of odd degree.
\end{remark}

We denote by $\rlie_1\alg$ the category of shifted restricted Lie algebras.

\begin{definition}
A $\W_1$-algebra is a  graded vector space over $\FF_p$ equipped with
\begin{itemize}
\item A degree $-1$ Lie bracket
\[[-,-]\colon V^i\otimes V^j\to V^{i+j-1}\]
\item A restriction
\[\xi\colon V^i\to V^{pi-p+1}\]
\item An additional linear map
\[\zeta\colon V^i\to V^{pi-p+2}\]
(if $p=2$, this map does not exist)
\item A product
\[V^i\otimes V^j\to V^{i+j}\]
\end{itemize}

such that the following axioms are satisfied.
\begin{enumerate}
\item The Lie bracket and $\xi$ make $V$ into a shifted restricted Lie algebra.
\item The product is bilinear and graded commutative.
\item The operation $\zeta$ vanishes on even degree elements.
\item We have the formula
\[[x,\zeta y]=0\]
\item The bracket is a derivation with respect to the product in each variable.
\[[x,yz]=[x,y]z+(-1)^{|y|(|x|+1)}[x,z]y\]
\item The operations $\xi$ and $\zeta$ each satisfy a Cartan formula
\[\xi(xy)=x^p\xi(y)+\xi(x)y^p+\sum\Gamma_{i,j}x^iy^j,\]
\[\zeta(xy)=\zeta(x)y^p-x^p\zeta(y)\]
where the term $\Gamma_{i,j}$ are described on page 335 of \cite{cohenladamay}.
(If $p=2$ the Cartan formula is
\[\xi(xy)=x^2\xi(y)+\xi(x)y^2+x[x,y]y.)\]
\end{enumerate}
\end{definition}

Let us denote by $\W_1\alg$ the category of $\W_1$-algebras and by
\[F^{\W_1}\colon \cat{grVect}_{\FF_p}\to \W_1\alg\]
the free $\W_1$-algebra monad.

\begin{theorem}[Cohen]
Let $C$ be a cochain complex of $\FF_p$-vector spaces. Let $E_2$ denote the free $C_*(E_2,\FF_p)$-algebra monad. There is a natural isomorphism
\[H^*(E_2(C))\cong F^{\W_1}(H^*(C))\]
In particular, the cohomology of a dg-$E_2$-algebra is naturally a $\W_1$-algebra.
\end{theorem}

\begin{proof}
This is done in \cite[Theorem III.3.1]{cohenladamay} for $E_2$-algebras coming from $E_2$-spaces and extended to any $E_2$-algebra in \cite[Theorem 16.4]{galatiuscellular}.
\end{proof}

By standard abstract nonsense, the restriction functor
\[\W_1\Alg\to \rlie_1\Alg\]
has a left adjoint that we shall denote by $F_{\rlie_1}^{\W_1}$. 

We shall need an explicit description of this left adjoint. Let $p$ be an odd prime, given a graded vector space, $V$ over $\FF_p$, we denote by $\zeta V$ the graded vector space
\[\zeta V=\bigoplus_{i \textrm{ odd}}s^{i-pi+p-2}V^i.\]
There is an operation 
\[\zeta\colon V\to\zeta V\]
taking $v\in V^i$ with $i$ odd to the element $v$ in the summand $s^{i-pi+p-2}V^i$ of $\zeta V$. This operation is not a map of graded vector spaces, instead it behaves with respect to the degree as the operation $\zeta$ in a $\W_1$-algebra. We define $\sym_{\zeta}$ to be the following functor from the category of graded vector spaces to itself
\[V\mapsto \sym(V\oplus\zeta V).\]

\begin{proposition}\label{prop : free functor}
Let $p$ be an odd prime. The composed functor
\[\rlie_1\alg\xrightarrow{F_{\rlie_1}^{\W_1}}\W_1\alg\xrightarrow{\rm{forget}}\cat{grVect}_{\FF_p}\]
is isomorphic to $\mathfrak{g}\mapsto\sym_{\zeta}(\mathfrak{g})$. A similar proposition holds over $\FF_2$ if we replace $\sym_{\zeta}$ by $\sym$.
\end{proposition}

\begin{proof}
Both functors of $\mathfrak{g}$ preserve ordinary sifted colimits. Moreover, $\rlie_1\alg$ is an algebraic category, i.e., it is the completion of its subcategory of free algebras on a finite dimensional vector space under ordinary sifted colimits. Therefore, it suffices to prove that both functors coincide on the category of free shifted restricted Lie algebras on a finite dimensional graded vector space. 

If $\mathfrak{g}$ is the free restricted Lie algebra on $V$, then a basis of $\mathfrak{g}$ is given by symbols $\xi^il_\alpha$ where the symbols $l_\alpha$ form an arbitrary basis of the free Lie algebra on $V$ and the exponent $i$ is arbitrary if $l_\alpha$ is of even degree and is zero otherwise. On the other hand, $F_{\rlie_1}^{\W_1}(\mathfrak{g})\cong F^{\W_1}(V)$ is explicitly described in \cite[page 227]{cohenladamay}. It is the free commutative algebra on elements of the form $\zeta^\epsilon\xi^il_\alpha$ where $l_\alpha$ and $\xi^i$ are as before and $\epsilon\in\{0,1\}$ and is required to be $0$ if $\xi^il_\alpha$ is of even degree. The result follows from this explicit description.

The case $p=2$ is similar.
\end{proof}

\begin{remark}\label{rem : even degree}
Observe that if $\mathfrak{g}$ is concentrated in even degrees (in which case the Lie bracket and $\xi$ must be zero for degree reasons), then $F_{\rlie_1}^{\W_1}(\mathfrak{g})$ is simply $\sym(\mathfrak{g})$ with trivial operations $[-,-]$, $\xi$ and $\zeta$.
\end{remark}

\begin{proposition}\label{prop : free is derived}
Let $V$ be a graded vector space, then 
\[\L \sym_{\zeta}(V)\simeq \sym_{\zeta}(V)\]
and 
\[\L \sym(V)\simeq \sym(V)\]
\end{proposition}

\begin{proof}
This proposition holds for any functor $F\colon \cat{grVect}_k\to\cat{grVect}_k$ that preserves filtered colimits. Indeed, $\L F$ coincides with $F$ on finite dimensional vector spaces by construction, moreover, both $\L F$ and $F$ preserve filtered colimits, it follows that they must coincide on any graded vector space.
\end{proof}

\section{Computation of the obstruction group}

\begin{definition}
We call a bigraded abelian group even (resp. odd) if it is concentrated in bidegrees $(s,t)$ such that $s+t$ is even (resp. odd).
\end{definition}

\begin{lemma}\label{lemm : Hochschild}
Let $A=\Lambda(V)$ be the exterior algebra on a graded vector space $V$ concentrated in odd degrees and of finite total dimension. Let $M$ be an $A$-module concentrated in odd degrees and finite-dimensional in each degree. Since $A$ is commutative, we may view $M$ as an $A$-bimodule. Then 
\[\AQ_{\rm{Ass}\alg}^{*,*}(A,M)\]
is even.
\end{lemma}

\begin{proof}
 There is an isomorphism
\begin{align*}
\AQ^{s,*}(A,M)&\cong \rm{Der}(A,M)\;\rm{if}\;s=0\\
              &\cong \rm{HH}^{s+1}(A,M)\;\rm{if}\;s>0
\end{align*}
Since $\rm{Der}(A,M)\subset \Hom_k(A,M)$ is concentrated in even degrees, it suffices to prove that $\rm{HH}^{*,*}(A,M)$ is odd.

There exists a unique map of commutative algebras
\[A\to A\otimes A\]
sending $v\in V$ to $v\otimes 1-1\otimes v$. Indeed, if $k$ is of characteristic different from $2$, then $\Lambda(V)\cong\sym(V)$ is free and if $k$ is of characteristic $2$, then we do indeed have
\[(v\otimes 1-1\otimes v)^2=0\]
By \cite[Theorem X.6.1]{cartaneilenberg}, we have an isomorphism
\[\rm{HH}^*(A,M)\cong\rm{Ext}^*_A(k,\tilde{M})\]
where $\tilde{M}$ is $M$ equipped with the $A$-module structure induced by restriction of scalars along the map $A\to A\otimes A$ described above. By definition of the $A$-bimodule structure on $M$, the $A$-module $\tilde{M}$ is then simply the $k$-vector space $M$ with the trivial $A$-module structure. So, using the finite dimensionality of $M$, we have
\[\rm{Ext}_A^*(k,\tilde{M})\cong \rm{Ext}^*_A(k,k)\otimes_kM\]
Since $M$ is odd, it suffices to prove that $\rm{Ext}^*_A(k,k)$ is even. If $V$ is one-dimensional so that $A=k[x]/x^2$. Then a free resolution of $k$ as an $A$-module is given by
\[k\leftarrow{A}\xleftarrow{\times x} s^{|x|}A \xleftarrow{\times x}s^{2|x|}A\xleftarrow{\times x}\ldots\]
It follows that the bigraded abelian group $\rm{Ext}^{*,*}_A(k,k)$ is even. For a general $V$, $\rm{Ext}_A(k,k)$ is a tensor product of finitely many bigraded abelian groups of this form, therefore the answer is still even.
\end{proof}

\begin{lemma}\label{lemm : main computation}

Let $V$ be a graded vector space concentrated in even degrees viewed as a shifted restricted Lie algebra with trivial bracket and restriction. Let $M$ be an abelian group object in the slice category $\W_1\Alg/F_{\rlie_1}^{\W_1}(V)$ whose underlying graded vector space is even. Then $\AQ^{*,*}_{\W_1}(F_{\rlie_1}^{\W_1}(V),M)$ is even.
\end{lemma}

\begin{proof}
By Proposition \ref{prop : base change}, we have an isomorphism
\[\AQ^*_{\W_1}(F_{\rlie_1}^{\W_1}(V),M)\cong \AQ^*_{\rlie_1}(V,M)\]
There is an adjunction
\[U\colon\rm{rLie}_1\cat{Alg}\leftrightarrows \rm{Ass}\cat{Alg}\colon\rm{forget} \]
between shifted restricted Lie algebras and associative algebras.  The right adjoint takes an associative algebra $A$, and sends it to $sA$ with Lie bracket given by
\[[sx,sy]=s(xy-(-1)^{|x||y|}yx)\]
and restriction given by
\[\xi(sx)=s(x^p)\]
(the notation $sx$ represents the element $x\in A$ viewed as an element of $sA$). The left adjoint $U$ sends $\mathfrak{g}$ to its universal enveloping algebra defined as the quotient $T(s^{-1}\mathfrak{g})/I$ where $I$ is the two-sided ideal generated by elements of the form $s^{-1}[g,h]-(s^{-1}g)(s^{-1}h)+(-1)^{(|g|+1)(|h|+1)}(s^{-1}h)(s^{-1}g)$ and $(s^{-1}g)^p-s^{-1}\xi(g))$ for any $g$ and $h$ homogeneous elements of $\mathfrak{g}$. This adjunction appears in \cite[Chapter V, Theorem 12]{jacobson} for restricted Lie algebra with bracket of degree zero.

In particular, one easily checks from the above formula that the universal enveloping algebra of the shifted restricted Lie algebra $V$ is simply the exterior algebra $\Lambda(s^{-1}V)$ regardless of the characteristic of the field. Then we can use again Proposition \ref{prop : base change}, and find an isomorphism
\[\AQ^*_{\rlie_1}(V,M)\simeq\AQ_{\rm{Ass}}^*(\Lambda(s^{-1}V),s^{-1}M) \]
which is even by the previous lemma. 
\end{proof}

\begin{remark}
The isomorphism between restricted Lie cohomology and associative cohomology of the universal enveloping algebra appears as \cite[Theorem 14.2]{dokasquillencohomologydividedpower}, although in their case, the restricted algebra structure is not shifted. The idea of reducing André-Quillen cohomology of $\W_1$-algebras to André-Quillen cohomology of restricted Lie algebras  was also explored in \cite{richterspectral} at the prime $2$. See for example \cite[Proposition 7.4]{richterspectral}.
\end{remark}

\begin{theorem}\label{theorem : main}

Let $A$ be a dg-$E_2$-algebra over $\FF_p$ such that $H^*(A)=\rm{Sym}(V)$ with $V$ a finite dimensional graded vector space concentrated in even degrees. Let $B$ be a dg-$E_2$-algebra also concentrated in even degrees and degreewise finite dimensional. Then
\begin{enumerate}
\item for any map of $\FF_p$-algebras
\[f\colon H^*(A)\to H^*(B)\]
there is a map in the homotopy category of $E_2$-algebras
\[\tilde{f}\colon A\to B\]
such that $H^*(\tilde{f})=f$.
\item Any $E_2$-algebra whose cohomology ring is isomorphic to the cohomology ring of $A$ must be quasi-isomorphic to $A$.
\end{enumerate}
\end{theorem}

\begin{proof}
First observe that (2) follows easily from (1). Thanks to Remark \ref{rem : even degree}, a map of $\FF_p$-algebras
\[f\colon H^*(A)\to H^*(B)\]
is automatically a map of $\W_1$-algebras. We may thus use the spectral sequence of \cite[Theorem 4.5]{johnsonnoel} computing the mapping space $\rm{map}_{E_2-\alg}(A,B)$. The relevant obstructions to lifting $f$ live in $E_2^{t,t-1}$. Thanks to \cite[Theorem B]{johnsonnoel}, this group can be identified with $\AQ^{t}_{\W_1}(A,s^{t-1}H^*(B))$. The result then follows from the previous lemma.
\end{proof}

\begin{remark}
A statement of this form is usually called \emph{intrinsic formality}. We are claiming that, up to quasi-isomorphisms, there is a unique $E_2$-algebra whose cohomology ring is a given polynomial algebra on even degree classes.
\end{remark}

\begin{remark}
A similar theorem was obtained by Devalpurkar (see \cite[Lemma 2.1.10]{devalapurkarku}) under the slightly stronger hypothesis that the $E_2$-structure can be promoted to an $E_3$-structure.
\end{remark}

In the next section, we will push this result to certain diagrams of $E_2$-algebras but we can already give one application originally due to Bayındır and Moulinos  (see \cite[Theorem 1.3]{bayindirmoulinos}).

\begin{theorem}[Bayındır, Moulinos]\label{theorem : bokstedt} 
Let $H\FF_p$ denote the Eilenberg-MacLane spectrum of $\FF_p$. There is a weak equivalence of $E_2$-algebras over $H\FF_p$~:
\[\rm{THH}(H\FF_p)\simeq H\FF_p\otimes\Sigma^{\infty}_+\Omega S^3.\]
\end{theorem}

\begin{proof}
Since $S^3\cong SU(2)$ is a Lie group, the space $\Omega S^3$ is a 2-fold loop space. Therefore, both sides of this equation are $E_2$-algebras in $H\FF_p$-modules so they can be viewed as dg-$E_2$-algebras. They also have isomorphic homotopy rings given by a polynomial ring on one generator of (homological) degree $2$.
\end{proof}

\begin{remark}
In fact Bayındır and Moulinos prove that for any polynomial ring $R$ over $\FF_p$ on one even degree class, there is a unique $E_2$-algebra with $R$ as its homotopy ring (see \cite[Theorem 2.1]{bayindirmoulinos}). This result also follows from Theorem \ref{theorem : main}.  The proof given in \cite{bayindirmoulinos} is also obstruction theoretic and is based on the Postnikov tower of $E_2$-ring spectra. Their obstruction groups are André-Quillen cohomology groups over $E_2$ (instead of over $\W_1$ for us) and we believe that they are in general more difficult to compute.
\end{remark}

\begin{remark}
There is a long literature devoted to the ring spectrum $\rm{THH}(H\FF_p)$. B\"okstedt's calculation shows that the two ring spectra of Theorem \ref{theorem : bokstedt} have isomorphic homotopy rings (see \cite{bokstedttopological}). This implies that the underlying $H\FF_p$-modules are weakly equivalent. A different approach to this result using a description of $H\FF_p$ as a Thom spectrum is given in \cite{blumbergtopological}. It was observed in \cite[Theorem 1.1]{krausenikolaus} that B\"okstedt's calculation can be refined to prove an equivalence of $E_1$-algebras over $H\FF_p$~:
\[\rm{THH}(H\FF_p)\simeq H\FF_p\otimes\Sigma^{\infty}_+\Omega S^3.\]
\end{remark}

For the sake of completeness, we state and sketch the proof of the characteristic zero analogue of the above theorem.

\begin{theorem}\label{theorem : main char zero}

Let $A$ be a dg-$E_2$-algebra over $\QQ$ such that $H^*(A)=\rm{Sym}(V)$ with $V$ a finite dimensional graded vector space concentrated in even degrees. Let $B$ be a dg-$E_2$-algebra also concentrated in even degrees and degreewise finite dimensional. Then
\begin{enumerate}
\item for any map of $\QQ$-algebras
\[f\colon H^*(A)\to H^*(B)\]
there is a map in the homotopy category of $E_2$-algebras
\[\tilde{f}\colon A\to B\]
such that $H^*(\tilde{f})=f$.
\item Any $E_2$-algebra whose cohomology ring is isomorphic to the cohomology ring of $A$ must be quasi-isomorphic to $A$.
\end{enumerate}
\end{theorem}

\begin{proof}
This is completely analogous to the proof of Theorem \ref{theorem : main} except that the monad $\W_1$ has to be replaced by the Gerstenhaber monad. We obtain exactly as above an isomorphism
\[\AQ_{\rm{Ger}}(H^*(A),H^*(B))\cong\AQ_{\rm{Ass}}(\Lambda(s^{-1}V),s^{-1}H^*(B)).\]
Moreover, the right hand side is even thanks to Lemma \ref{lemm : Hochschild} (which does not depend on the characteristic of the field).
\end{proof}

\begin{remark}
The analogous theorem with $E_2$ replaced by $E_\infty$ is also true and well-known. Indeed, in that case, $A$ can be strictified to a commutative algebra. Then we can produce a quasi-isomorphism 
\[H^*(A)\to A\]
by sending each generator to a choice of cocycle representing it.
\end{remark}

\section{Diagrams of $E_2$-algebras}

\subsection{Main theorem}

Let $\cat{C}$ be a quasi-algebraic category and let $I$ be a small category. In this situation $\cat{C}^I$ is also a quasi-algebraic category. Given $c\colon I\to\cat{C}$ and $m\in\cat{Ab}(\cat{C}^I/c)$, we denote by $\AQ_{\cat{C},I}(c,m)$ the corresponding Quillen cohomology.

We start with the following proposition for which we could not find a reference.
 
\begin{proposition}
Let $I$ be a small category, let $A$ be an associative algebra in $\cat{grVect}^I$ and let $M$ be an $A$-$A$-bimodule in $\cat{grVect}^I$. Then 
\[\AQ^*_{\rm{Ass},I}(A,M)\cong \rm{Ext}^*_{A\otimes A^\op}(\Omega_A,M)\] 
where $\Omega_A$ is the bimodule of associative differentials defined by the following exact sequence
\[0\to \Omega_A\to A\otimes A\xrightarrow{m}A\]
where $m$ denotes the multiplication map.  
\end{proposition}

\begin{proof}
This proposition is very classical if $I=[0]$. The object $\Omega_A$ is the result of applying the abelianization functor to $\id_A\colon A\to A$ viewed as an object of $\cat{AssAlg}^I/A$. By definition of Quillen cohomology, this proposition boils down to proving an equivalence
\[\mathbb{L}\mathrm{Ab}(\id_A)\xrightarrow{\simeq} \Omega_A=\mathrm{Ab}(\id_A)\]
where $\rm{Ab}\colon s\cat{AssAlg}/A\to s\cat{Mod}_{A\otimes A^\op}$ is the abelianization functor. Moreover, it is enough to prove this equivalence pointwise.

For any $i$ in $I$ and any category $\cat{C}$, let us denote by $ev_i$ the evaluation functor $\cat{C}^I\to \cat{C}$. According to \cite[Theorem 4.7]{franklandbehavior}, there is a commutative diagram of Quillen left adjoints (the theorem applies since $ev_i$ preserves projective objects):
\[\xymatrix{
s(\cat{AssAlg}^I/A)\ar[d]_{ev_i}\ar[r]^{\rm{Ab}}&s\cat{Mod}_{A\otimes A^\op}\ar[d]^{ev_i}\\
s(\cat{AssAlg}/A(i))\ar[r]_{\rm{Ab}}&s\cat{Mod}_{A(i)\otimes A(i)^\op}
}
\]
Applying this to $\id_A$ and deriving the functors, we obtain a weak equivalence
\[(\mathbb{L}\Omega_A)(i)\simeq \mathbb{L}\Omega_{A(i)}\]
Now, using the case $I=[0]$ of the proposition, we find
\[(\mathbb{L}\Omega_A)(i)\simeq \Omega_{A(i)}=\Omega_A(i)\]
as desired.
\end{proof}

\begin{proposition}
Let $V\colon I\to\cat{grVect}$ be a diagram taking values in finite dimensional graded vector spaces concentrated in odd degrees. Let $M\colon I\to \cat{grVect}$ be a $\Lambda(V)$-module satisfying pointwise the conditions of Lemma \ref{lemm : Hochschild}. We moreover assume that $M$ viewed as a diagram of $k$-vector spaces is injective. Then $\AQ_{\rm{Ass},I}(\Lambda(V),M)$ is even.
\end{proposition}

\begin{proof}
In order to prove this proposition, we shall first construct a spectral sequence computing the groups $\AQ_{\rm{Ass},I}(\Lambda(V),M)$.

For $A$ an associative algebra in $\cat{grVect}^I$ and $M$ an $A$-bimodule, using the previous proposition, we have an isomorphism
\[\AQ^*_{\rm{Ass},I}(A,M)=\pi_{-*}\RR\Hom_{A\otimes A^\op}(\Omega_A,M).\]

In order to compute this derived Hom, we can use first the bar resolution in the category of bimodules and obtain
\[\RR\Hom_{A\otimes A^\op}(\Omega_A,M)=\mathrm{holim}_{[n]\in \Delta}\RR\Hom_{\cat{grVect}^I}((A\otimes A^\op)^{\otimes n}\otimes\Omega_A,M)\]

Now, in general, if $\cat{C}$ is an $\infty$-category and $I$ a small category, we have
\[\rm{map}_{\cat{C}^I}(F,G)\simeq\rm{holim}_{\Delta}\left([s]\mapsto\prod_{i_0\to i_1\to\ldots\to i_s} \rm{map}(F(i_0),G(i_s))\right).\]
So we can write the derived Hom as the limit of the following double cosimplicial diagram
\[([s],[n])\mapsto \left(\prod_{i_0\to i_1\to\ldots\to i_s}\RR\Hom_{\cat{grVect}}((A\otimes A^\op)^{\otimes n}\otimes\Omega_A(i_0),M(i_s))\right)\]

Taking the limit in the $[n]$ direction, we get
\[\RR\Hom(\Omega_A^1,M)=\mathrm{holim}_{[s]\in \Delta}\left(\prod_{i_0\to i_1\to\ldots\to i_s}\RR\Hom_{A(i_0)\otimes A(i_0)^\op}(\Omega_{A(i_0)},M(i_s))\right)\]
As for any cosimplicial object, we obtain an associated Bousfield-Kan spectral sequence whose $E_1$ page is given by
\[E_1^{s,t}=\prod_{i_0\to i_1\to\ldots\to i_s}\AQ^{t}_{\rm{Ass}}(A(i_0),M(i_s)).\]
This spectral sequence converges as it can also be viewed as the spectral sequence of a third quadrant double complex (coming from the above bicosimplicial abelian group after taking the associated complex in both cosimplicial direction). From the comparison between Quillen cohomology and Hochschild cohomology, we see that the row $t=0$ is given by the cosimplicial abelian group
\[[s]\mapsto \prod_{i_0\to i_1\to\ldots\to i_s}\rm{Der}(A(i_0),M(i_s)),\]
while for $t>0$, we get
\[[s]\mapsto \prod_{i_0\to i_1\to\ldots\to i_s}\rm{HH}^{t+1}(A(i_0),M(i_s)).\]

So far we did not use anything about the specific situation and the above discussion would apply to any pair $(A,M)$. Using the computation of Lemma \ref{lemm : Hochschild}, we see that the row $t>0$ of the $E_1$-page is of the form
\[[s]\mapsto \prod_{i_0\to i_1\to\ldots\to i_s}\Hom_k(F(i_0),M(i_s)),\]
where $F\colon I\to\cat{grVect}_k$ is the degreewise dual of the functor
\[i\mapsto \rm{Ext}_{A(i)}(k,k).\]
Likewise the $0$th row is simply given by the cosimplicial abelian group
\[[s]\mapsto \prod_{i_0\to i_1\to\ldots\to i_s}\Hom_k(V(i_0),M(i_s)).\]
(indeed there is an isomorphism $\rm{Der}(\Lambda(V),M)\cong\Hom_k(V,M)$). From this observation, using injectivity of $M$, we deduce that the $E_2$-page of the spectral sequence is concentrated on the column $s=0$ and
\[E_2^{0,-t}=\rm{ker}\left(d_1\colon \prod_{i\in I}\AQ^{t}(A(i),M(i))\to \prod_{f\colon i\to j}\AQ^{t}(A(i),M(j))\right).\]
In particular, we see that 
\[\AQ^{*,*}_I(A,M)\subset\prod_{i\in I}\AQ^{*,*}(A(i),M(i))\]
and is therefore even by Lemma \ref{lemm : Hochschild}.
\end{proof}

\begin{corollary}
Let $V\colon I\to\cat{grVect}_{\FF_p}$ be a diagram of finite dimensional graded vector space concentrated in even degrees. Let $M\colon I\to\cat{grVect}_{\FF_p}$ be an injective diagram concentrated in even degrees and equipped with the structure of a module over $F_{\rlie_1}^{\W_1}(V)$. Then $\AQ_{\W_1,I}^{*,*}(F_{\rlie_1}^{\W_1}(V),M)$ is even.
\end{corollary}

\begin{proof}
As in Lemma \ref{lemm : main computation}, we can reduce to showing that $\AQ_{Ass,I}(\Lambda(s^{-1}V),s^{-1}M)$ is even which is the content of the previous proposition.
\end{proof}

\begin{theorem}\label{theorem : main in family}
Assume that $i\mapsto A(i)$ is a diagram of differential graded $E_2$-algebras over $\FF_p$ such that
\begin{enumerate}
\item There is a diagram $V\colon I\to\cat{grVect}_{\FF_p}$ which is objectwise finite dimensional and concentrated in even degrees, and a natural isomorphism
\[H^*(A(i))\cong F_{\rlie_1}^{\W_1}(V(i)).\]
\item The diagram 
\[i\mapsto H^*(A(i))\]
is injective as an $I$-diagram in $\FF_p$-vector spaces.
\end{enumerate}
Then the diagram $A$ is formal as a diagram of $E_2$-algebras.
\end{theorem}

\begin{proof}
As in Theorem \ref{theorem : main}, we use obstruction theory. We check that the hypotheses of \cite[Theorem B]{johnsonnoel} hold. The category denoted $\mathscr{D}$ in \cite[Theorem B]{johnsonnoel} is simply the category of $I$-diagrams of graded vector spaces. Then condition (a) of \cite[Theorem B]{johnsonnoel} holds thanks to our injectivity assumption. The monad $T_{alg}$ in \cite[Theorem B]{johnsonnoel} is simply $F^{\W_1}$ applied objectwise. It follows that the relevant obstruction group is $\AQ_{\W_1,I}^{t,t-1}(F_{\rlie_1}^{\W_1}(V),H^*(A))$ which is zero thanks to the previous corollary.
\end{proof}

\subsection{Criterion for injectivity}

We borrow the following definition from Hovey (see \cite[Definition 5.1.1]{hoveymodelcategories})

\begin{definition}
A direct category is a category $I$ with a functor $\lambda\colon I\to(\mathbb{N},\leq)$ such that $\lambda(f)=\id$ if and only if $f=\id$.  
\end{definition}

\begin{proposition}
Let $I$ be a direct category. Consider a diagram $F\colon I^\op\to\cat{grVect}$. Then $F$ is injective if for all $i\in I$ the canonical map
\[F(i)\to\rm{lim}_{j\in I,\lambda(j)<\lambda(i)} F(j)\]
is surjective.
\end{proposition}

\begin{proof}
This is very similar to \cite[Proposition 5.1.4]{hoveymodelcategories}. Let us call the number $\lambda(i)$ the ``degree'' of the object $i$. Given an objectwise injective map $f\colon M\to N$ in $\cat{grVect}^{I^\op}$, we can lift a map $g\colon M\to F$  to a map $\tilde{g}\colon N\to F$ inductively on degree. Assuming the lift has been produced up to degree $n$, the next step is to find a lift on an object $i$ of degree $n+1$. This amounts to finding a diagonal filler in the following diagram
\[\xymatrix{
M(i)\ar[d]_{f(i)}\ar[r]^{g(i)}&F(i)\ar[d]\\
N(i)\ar[r]&\rm{lim}_{\lambda(k)<\lambda(i)}F(k)
}
\]
in which the bottom horizontal map is the composite
\[N(i)\to \rm{lim}_{\lambda(k)<\lambda(i)}N(k)\xrightarrow{\tilde{g}} \rm{lim}_{\lambda(k)<\lambda(i)}F(k)\]
This lifting problem has a solution since the left vertical map is injective while the right vertical map is surjective. Once these liftings have been chosen for each $i$ of degree $n+1$, they are automatically compatible since there are no non-identity maps between objects of the same degree. This completes the induction step.
\end{proof}

\begin{example}\label{remark : injective 1}
Let $I$ be the category with two objects $0$ and $1$ and a single non-identity map $0\to 1$. We can consider $I$ as a direct category with $\lambda(0)=0$, $\lambda(1)=1$. The proposition says that an arrow $f\colon M_1\to M_0$ viewed as an $I^\op$-diagram in $\cat{grVect}$ is injective if $f$ is surjective. Similarly, a span diagram
\[M\to N\leftarrow P\]
is injective if each of the maps is surjective.
\end{example}

\begin{example}\label{remark : injective 2}
Let $I$ be the poset of faces of a simplicial complex. Then $F\colon I^\op\to\cat{grVect}$ is injective if it is fat, in the sense of \cite[Definition 3.6]{notbohmray}. Indeed, we can view $I$ as a direct category with $\lambda\colon I\to\mathbb{N}$ the dimension function.
\end{example}

\section{Applications}

\subsection{Formality of $BG$ for some compact Lie groups}

\begin{theorem}\label{theorem : formality compact Lie group}
Let $G$ be a compact Lie group with maximal torus $T$ and assume that $p$ does not divide the order of $N_G(T)/T$. Then $C^*(BG,\FF_p)$ is $E_2$-formal.
\end{theorem}

\begin{proof}
Let us write $W=N_G(T)/T$. In this situation, by \cite[Theorem 1.5]{feshbach} there is a quasi-isomorphism of $E_2$-algebra
\[C^*(BG;\FF_p)\to C^*(BT,\FF_p)^W\simeq C^*(BT,\FF_p)^{hW}.\]
Similarly, there is an isomorphism of commutative algebras
\[H^*(BG;\FF_p)\to H^*(BT,\FF_p)^W.\]
Since $H^*(BT;\FF_p)$ is polynomial on even degree generators, then the result will hold if the formality quasi-isomorphism
\[C^*(BT;\FF_p)\simeq H^*(BT;\FF_p)\]
given by Theorem \ref{theorem : main} can be made $W$-equivariant. Since, by assumption, $p$ does not divide the order of $W$, it follows that any $\FF_p[W]$-vector space is injective as a $W$-diagram and thus the result follows from Theorem \ref{theorem : main in family}.
\end{proof}

\begin{remark}
The fact that, under these assumptions $C^*(BG,\FF_p)$ is formal as an $E_1$-algebra is a theorem of Benson and Greenlees (see \cite{bensonformality}).
\end{remark}

\begin{remark}
It was observed by Benson and Greenlees in \cite{bensonformality} that a compact Lie group satisfying the assumptions of the above theorem does not necessarily have polynomial cohomology. They give the example of the non-connected group $G=\ZZ/2\ltimes T^2$ with $\ZZ/2$ acting on the torus by inversion. In that case $H^*(BG,\FF_3)$ is not a polynomial algebra. It is given by the subalgebra $\FF_3[x^2,xy,y^2]$ of $H^*(BT^2,\FF_3)\cong\FF_3[x,y]$. Nevertheless $BG$ is $E_2$-formal over $\FF_3$ thanks to our theorem.
\end{remark}

\subsection{Formality of Davis-Januszkiewicz spaces}

For our next application, recall that, given a relative CW-complex $A\subset X$ and a simplicial complex $K$ with set of vertices $V$, we may form the polyhedral product $Z(K;(A,X))\subset X^V$ given as 
\[Z(K;(A,X))=\bigcup_{\sigma\in K}X^{\sigma}\times A^{V-\sigma}.\]
A case of particular interest is the case $A=pt$ and $X=BS^1$. The resulting space is then called a Davis-Januszkiewicz space (this was originally introduced in \cite{davisjanu}, see also \cite[Chapter 4]{buchstaberpanov}).

\begin{theorem}\label{theorem : formality DJ space}
Let $G$ be a compact Lie group such that the $\FF_p$-cohomology of $BG$ is polynomial on even degree generators. Then $C^*(Z(K;(pt,BG)),\FF_p)$ is $E_2$-formal for any simplicial complex $K$. Moreover we have an isomorphism of algebras
\[H^*(Z(K;(pt,BG)),\FF_p)\cong\lim(\sigma\mapsto H^*(BG^\sigma,\FF_p)).\]
\end{theorem}

\begin{proof}
This is an application of Theorem \ref{theorem : main in family}. In this case the diagram
\[\sigma\mapsto H^*(BG^\sigma,\FF_p)\]
is injective by \cite[Lemma 3.8]{notbohmray} and Example \ref{remark : injective 2}. It follows that the diagram $K^{\op}\to\alg_{E_2}$
\[\sigma\mapsto C^*(BG^\sigma,\FF_p)\]
is formal, therefore, we have a quasi-isomorphism of $E_2$-algebras
\[\rm{holim}(\sigma\mapsto H^*(BG^\sigma,\FF_p))\simeq \rm{holim}(\sigma\mapsto C^*(BG^\sigma,\FF_p))\simeq C^*(Z(K;(pt,BG));\FF_p). \]
\end{proof}

\begin{remark}
The case $G=S^1$ is a theorem of Matthias Franz (see \cite{franzhomotopy}). Franz phrases his result using the notion of ``homotopy Gerstenhaber algebras'' instead of $E_2$-algebras. The operad governing homotopy Gerstenhaber algebras is the complexity $2$ suboperad of the surjection operad (see \cite[Subsection 3.1]{franzhomotopygerst}) which is known to be a model for $E_2$ (see \cite[Section 4]{mccluresmith} or \cite[Subsection 1.6]{bergerfresse}) so the two results should be equivalent. Let us also mention that in that case, the cohomology of $Z(K;(pt,BS^1))$ can be computed and is given by the Stanley-Reisner algebra associated to the simplicial complex $K$. Another related result is the main theorem of \cite{notbohmrayII} that proves that the homotopy type of the Davis-Januszkiewicz space is determined by the integral cohomology ring if this ring is a complete intersection after extending the scalars to $\mathbb{Q}$.
\end{remark}

\begin{remark}
The polyhedral product construction can be extended to any map $A\to X$ (not just relative CW-complexes) by replacing the definition above by the homotopy colimit of the diagram 
\[\sigma\mapsto X^{\sigma}\times A^{V-\sigma}.\]
In particular, if $G$ is a compact Lie group, we may form $Z(K;(G,pt))$. This spaces inherits a $G^V$-action and we have
\[Z(K;(G,pt))_{hG^V}\simeq (Z(K;(pt,BG)))\]
by \cite[Lemma 2.3.2]{denhamsuciu}.
\end{remark}

\subsection{Multiplicative collapse of some Eilenberg-Moore spectral sequences}

Let $X\to B\leftarrow Y$ be a diagram of spaces. For each $b\in B$, we write $X_b:=\{b\}\times^h_BX$, the homotopy fiber of $X\to B$ over $b$. Recall that the cohomological Eilenberg-Moore spectral has as its $E_2$-page the group $\rm{Tor}^{H^*(B,\FF_p)}(H^*(X,\FF_p),H^*(Y,\FF_p))$. It is known to converge to $H^*(X\times_B^hY,\FF_p)$ under the following two assumptions.
\begin{enumerate}
\item For each $b\in B$ and each integer $j$, the vector space $H^j(X_b,\FF_p)$ is finite dimensional.
\item For each $b\in B$ and each integer $j$, the action of $\pi_1(B,b)$ on $H^j(X_b)$ is nilpotent (that is there is a finite filtration of $H^j(X_b)$ that is compatible with the action and such that the induced action on the associated graded pieces is trivial).
\end{enumerate}

\begin{theorem}\label{theorem : formality and EMSS}
Let $X\to B\leftarrow Y$ be a diagram of spaces. Assume that
\begin{enumerate}
\item The cohomology of all three spaces with $\FF_p$ coefficients is a polynomial algebra on finitely many even degree generators.
\item The maps $H^*(B,\FF_p)\to H^*(X,\FF_p)$ and $H^*(B,\FF_p)\to H^*(Y,\FF_p)$ are surjective and send generators to linear combinations of generators.
\end{enumerate}
Then, the Eilenberg-Moore spectral sequence collapses multiplicatively. In particular, if the spectral sequence converges (for example if the conditions recalled above are satisfied) there is an isomorphism of $\FF_p$-algebras
\[\rm{Tor}^{H^*(B,\FF_p)}(H^*(X,\FF_p),H^*(Y,\FF_p))\cong H^*(X\times^h_BY,\FF_p).\]
\end{theorem}

\begin{proof}
The surjectivity assumption ensures that the diagram
\[H^*(X,\FF_p)\leftarrow H^*(B,\FF_p)\to H^*(Y,\FF_p)\]
is injective by Example \ref{remark : injective 1}. From Theorem \ref{theorem : main in family}, we obtain that the diagram of $E_2$-algebras
\[C^*(X,\FF_p)\leftarrow C^*(B,\FF_p)\to C^*(Y,\FF_p)\]
is formal. It follows that there is a quasi-isomorphism of $E_1$-algebras
\[H^*(X,\FF_p)\otimes_{H^*(B,\FF_p)}^{\L}H^*(Y,\FF_p)\simeq C^*(X,\FF_p)\otimes_{C^*(B,\FF_p)}^{\L}C^*(Y,\FF_p).\]
Under the Eilenberg-Moore convergence assumption, the right-hand side is quasi-isomorphic to $C^*(X\times^h_BY,\FF_p)$.
\end{proof}

\begin{remark}
Arguably the most interesting case of application of the previous theorem is when $X$, $B$ and $Y$ are classifying space of compact Lie groups. In that case this Theorem is a weaker version of the main results of \cite{franzhomogeneous,carlsontwosided}. Indeed the main theorem in those papers does not have our second assumption, also these papers allow for the coefficient ring to be a PID instead of a field. On the other hand, those papers assume that $2$ is invertible in their coefficients ring and we are able to say something also in the case $p=2$.
\end{remark}

\begin{remark}
There is a long tradition of collapse results for the Eilenberg-Moore spectral sequence especially the one computing the cohomology of a homogeneous space $G/H$ with $G$ and $H$ Lie groups. Previously to the work of Franz and Carlson mentioned in the previous remark, we refer the reader to \cite{baumcohomology}, \cite{gugenheimmay} and \cite{wolfcohomology} for additive collapse results. Multiplicative collapse was proven over the reals by Cartan and over any field but under quite strong assumptions by Borel (see \cite[Théorème 6]{cartan} and \cite[Proposition 30.2]{borel}).  
\end{remark}

\begin{example}
Consider the complex Stiefel manifolds $V_{n,k}=U(n)/U(n-k)$. This may be viewed as the homotopy fiber of the map 
\[BU(n-k)\to BU(n)\]
induced by the standard inclusion. For any prime, the resulting Eilenberg-Moore spectral sequence
\[\mathrm{Tor}^{H^*(BU(n),\FF_p)}(H^*(BU(n-k),\FF_p),\FF_p)\Rightarrow H^*(V_{n,k},\FF_p)\]
satisfies the assumption of Theorem \ref{theorem : formality and EMSS} above and thus, we get a multiplicative computation of $H^*(V_{n,k},\FF_p)$ for any value of $p$. This result is classical and originally due to Borel (see \cite[Section 9]{borel}).
\end{example}

\begin{remark}
In contrast with the previous example, consider the diagonal inclusion $U(1)\to U(n)$. The quotient is $PU(n)$. By \cite[Section 8, Example 4]{baumcohomology}, the induced map in cohomology
\[H^*(BU(n),\FF_2)\to H^*(BU(1),\FF_2)\]
is surjective if and only if $n$ is odd. Unfortunately, our theorem does not imply collapse of the Eilenberg-Moore spectral sequence in those cases since the map does not send the generators to linear combination of generators. In any case, if $n\equiv 2\;\rm{mod}\;4$ this spectral sequence is known to have multiplicative extensions. In particular, if $n=2$, there is a homeomorphism $PU(2)\cong \mathbb{RP}^3$ and it is observed in \cite[Remark 12.9]{franzhomogeneous} that the algebra
$\rm{Tor}^{H^*(BU(2),\FF_2)}(H^*(BU(1),\FF_2),\FF_2)$ contains a non-zero element in degree $1$ that squares to zero.
\end{remark}

\begin{corollary}
Let $(X,x)$ be a based space with $H^*(X,\FF_p)\cong \sym(V)$ with $V$ finite dimensional an concentrated in even positive degrees. Then there is an isomorphism of Hopf algebras
\[H^*(\Omega X,\FF_p)\cong \Lambda(s^{-1}V)\]
\end{corollary}

\begin{proof}
The previous theorem gives us an isomorphism of algebras
\[H^*(\Omega X,\FF_p)\cong \rm{Tor}^{\sym(V)}(\FF_p,\FF_p)=\Lambda(s^{-1}V)\]
On the other hand, by Theorem \ref{theorem : main in family}, the diagram
\[C^*(pt,\FF_p)\leftarrow C^*(X,\FF_p)\rightarrow C^*(pt,\FF_p) \]
is $E_2$ (and hence $E_1$) formal. Equivalently, the augmented algebra $C^*(X,\FF_p)$ is $E_1$-formal. This implies that there is an isomorphism of coalgebras
\[H^*(\Omega X,\FF_p)\cong H^*(\rm{Bar}(H^*(X,\FF_p))\]
where $\rm{Bar}$ denotes the bar construction of an augmented algebra~:
\[\rm{Bar}(A):=\FF_p\otimes^{\L}_A\FF_p.\]
\end{proof}

\begin{remark}
As we mentioned in the proof of the corollary, the statement about the coalgebra structure only requires $E_1$-formality. This holds if the cohomology is polynomial without the evenness assumption (see \cite{halperinrite}). On the other hand, the statement about the algebra structure is not true if we only assume $E_1$-formality. As an example of this phenomenon consider $X=K(\ZZ/2,2)$. Then we have
\[H^*(X,\FF_2)=\FF_2[x_{2^n+1},n\geq 0]\]
with $x_2$ denoting the fundamental class and 
\[x_{2^n+1}=\rm{Sq}^{2^{n-1}}\ldots\rm{Sq}^1 x_2.\]
Then $C^*(X,\FF_2)$ is $E_1$-formal since its cohomology is polynomial. 

On the other hand, $H^*(\Omega X,\FF_2)=H^*(\mathbb{RP}^\infty,\FF_2)=\FF_2[y]$ with $|y|=1$ and with the Hopf algebra structure given by
\[\Delta(y)=y\otimes 1+1\otimes y.\]
We observe that, as a coalgebras, there is indeed an isomorphism
\[H^*(\Omega X,\FF_2)\cong\bigotimes_{n\geq 0}\Lambda[y_{2^n}]=H^*(\rm{Bar}(H^*(X,\FF_2)))\]
but this isomorphism is not compatible with the algebra structure. 

In fact, it can be shown that $C^*(X,\FF_2)$ is not $E_2$-formal. Indeed, for any space $Y$, the operation $\xi$ of the $\W_1$-structure on $H^*(Y,\FF_2)$ is simply given by $x\mapsto \rm{Sq}_1(x)=\rm{Sq}^{|x|-1}(x)$. It follows that an $E_2$-formal space must have a trivial operation $\rm{Sq}_1$. This is not the case for $X$. 
\end{remark}

\bibliographystyle{amsalpha}
\bibliography{biblio}

\end{document}